\documentclass{article}

\usepackage{amsmath}
\usepackage{amssymb}
\usepackage{amsthm}
\usepackage{enumerate}

\newcommand{\paren}[1]{\left( #1 \right)}
\newcommand{\angles}[1]{\left\langle #1 \right\rangle}
\newcommand{\curly}[1]{ \left\{ #1 \right\}}
\newcommand{\abs}[1]{\left| #1 \right|}
\newcommand{\brac}[1]{\left[ #1 \right]}
\newcommand{\Q}{\mathbb{Q}}
\newcommand{\forces}{\Vdash}

\newtheorem{theorem}{Theorem}
\newtheorem{defn}[theorem]{Definition}
\newtheorem{claim}[theorem]{Claim}
\newtheorem{lem}[theorem]{Lemma}
\newtheorem{question}[theorem]{Question}
\newtheorem{corollary}[theorem]{Corollary}

\DeclareMathOperator{\otp}{otp}
\DeclareMathOperator{\cf}{cf}
\DeclareMathOperator{\dom}{dom}
\DeclareMathOperator{\Lim}{Lim}

\title{Chang's Conjecture with $\square_{\omega_1, 2}$ from an $\omega_1$-Erd\H{o}s Cardinal}
\author{Itay Neeman\footnote{This material is based upon work supported by the National Science Foundation under Grants No. DMS-1363364 and DMS-1764029.}\; and John Susice\footnotemark[\value{footnote}]}

\begin{document}
\maketitle{}

\begin{abstract}
	Answering a question of Sakai \cite{cc_weak_square}, we show that the existence of an $\omega_1$-Erd\H{o}s 
cardinal suffices to obtain the consistency of Chang's Conjecture with $\square_{\omega_1, 2}$.  
By a result of Donder \cite{some_apps_of_core_model} this is best possible. 

We also give an answer to another question of Sakai relating to the incompatibility 
of $\square_{\lambda, 2}$ and $(\lambda^+, \lambda) \twoheadrightarrow (\kappa^+, \kappa)$
for uncountable $\kappa$. 
\end{abstract}

\section{Introduction}
Chang's Conjecture is a model-theoretic principle asserting a strengthening of 
the L\"owenheim-Skolem Theorem \cite{note_on_two_cardinal_problem}. 
Chang's Conjecture was originally shown to be consistent assuming the existence of a 
Ramsey cardinal by Silver (see \cite{evolution_large_cardinal_axioms}) and this 
assumption was later weakened to the existence of an $\omega_1$-Erd\H{o}s cardinal 
\cite{donder_levinski}. This result is best possible, since Chang's Conjecture implies 
that $\omega_2$ is $\omega_1$-Erd\H{o}s in the core model \cite{some_apps_of_core_model}. 

Chang's Conjecture is known to be incompatible with Jensen's square principle $\square_{\omega_1}$
(see \cite{walks_on_ordinals}) but was recently shown to be consistent with 
Schimmerling's square principle $\square_{\omega_1, 2}$ by Sakai 
\cite{cc_weak_square}, assuming the existence of a measurable cardinal. 
In light of this consistency upper bound, Sakai posed the following: 

\begin{question}
	What is the consistency strength of the conjunction of 
	Chang's Conjecture with $\square_{\omega_1, 2}$? 
\end{question}

In Corollary \ref{obvious_corollary} we show that 
the consistency of the given statement follows from the existence of an 
$\omega_1$-Erd\H{o}s cardinal, answering Sakai's question. 
Section \ref{prelims} of the paper will cover some basic preliminaries, 
such as the definition of the relevant square principle and large cardinal. 
In Section \ref{the_poset} we describe our forcing poset. In Silver's 
consistency proof, he used what is now called a Silver forcing poset--a 
modification of the Levy Collapse forcing which allows larger supports 
\cite{evolution_large_cardinal_axioms}. 
Cummings and Schimmerling \cite{indexed_squares} have introduced another 
variant of the Levy Collapse forcing which collapses inaccessible 
$\kappa$ to $\omega_2$ while simultaneously adjoining a square sequence. 
Our forcing will be a hybrid of these two posets--in other words it will be 
a ``Silverized'' Cummings-Schimmerling poset. 

Finally, in Section \ref{the_proof} we give the proof of our result, 
which is based on the methods of \cite{cc_weak_square} and 
\cite{donder_levinski}. 


In Section \ref{second_section} we investigate the relation between weak square principles and model theoretic 
transfer properties (i.e., generalizations of Chang's Conjecture) of the form 
$\paren{\lambda^+, \lambda} \twoheadrightarrow \paren{\kappa^+, \kappa}$ for $\kappa \geq \aleph_1$. 
Sakai proved the following: 

\begin{theorem}[Sakai, \cite{cc_weak_square}]
	Suppose that $\paren{\lambda^+, \lambda} \twoheadrightarrow \paren{\kappa^+, \kappa}$, 
	where $\kappa$ is an uncountable cardinal and $\lambda$ is a cardinal $> \kappa$. 
	Moreover, suppose that either of the following holds: 
	\begin{enumerate}[(I)]
		\item $\lambda^{< \lambda} = \lambda$
		\item $\kappa < \aleph_{\omega_1}$, and there are strictly more regular cardinals 
			in the interval $\brac{\aleph_0, \kappa}$ than in the interval 
		$(\kappa, \lambda]$. 
	\end{enumerate}
	Then $\square_{\lambda, \kappa}$ fails. 
	\label{sakais_constraint}
\end{theorem}
Although Theorem \ref{sakais_constraint} imposes substantial constraints on the interaction of weak 
square principles and model theoretic transfer properties, there are many instances 
where it does not apply. For example, it does not answer the question of whether 
$(\aleph_4, \aleph_3) \twoheadrightarrow (\aleph_2, \aleph_1)$ is incompatible with 
$\square_{\omega_3, 2}$ when $2^{\aleph_2} > \aleph_3$.

In light of these limitations, Sakai posed the following question: 
\begin{question}[Sakai, \cite{cc_weak_square}]
	Let $\kappa$ be an uncountable cardinal and $\lambda$ a cardinal $> \kappa$. 
	Does $\paren{\lambda^+, \lambda} \twoheadrightarrow \paren{\kappa^+, \kappa}$ imply 
	the failure of $\square_{\lambda, 2}$? 
\end{question}
We answer this question in the affirmative in Corollary \ref{zeroeth_big_corollary} (in fact we 
obtain the failure of $\square_{\lambda, \omega}$ under these hypotheses and more under slightly 
stronger hypotheses--see Corollaries \ref{first_big_corollary} and \ref{second_big_corollary}). 
Taking $\kappa = \aleph_1$, $\lambda = \aleph_3$ in this theorem shows that indeed 
$(\aleph_4, \aleph_3) \twoheadrightarrow (\aleph_2, \aleph_1)$ is incompatible 
with $\square_{\omega_3, 2}$, regardless of the value of $2^{\aleph_2}$.

\section{The Consistency of Chang's Conjecture and $\square_{\omega_1, 2}$ from an $\omega_1$-Erd\H{o}s Cardinal}
\label{consistency_section}

\subsection{Preliminaries} \label{prelims}

In the following, for any cardinal 
$\theta$ we denote by $H\!\paren{\theta}$ the collection of 
all sets whose transitive closure has size $< \theta$. We frequently 
confuse a structure and its underlying set. I.e., if $\mathcal{M} = 
\angles{M, \dots}$ is a structure and $\alpha$ is an ordinal, we 
write $\alpha \subseteq \mathcal{M}$ to mean $\alpha \subseteq M$. All structures 
we consider have at most countably many symbols in their signature. 

\begin{defn}
	\emph{Chang's Conjecture} is the assertion that for any 
	structure $\mathcal{N}$ with $\omega_2 \subseteq \mathcal{N}$, 
	there exists 
	$\mathcal{M} \preceq \mathcal{N}$ such that 
	$\abs{\mathcal{M}} = \aleph_1$ and 
	$\abs{\mathcal{M} \cap \omega_1} = \aleph_0$. 
\end{defn}
We observe that to verify Chang's Conjecture it suffices to verify 
it for models with underlying set $H\!\paren{\omega_2}$: 

\begin{claim}[Folklore]
	Suppose that for all structures $\mathcal{H} = 
	\langle H\!\paren{\omega_2}, \dots \rangle$ there exists 
	$\mathcal{M} \preceq \mathcal{H}$ of cardinality 
	$\aleph_1$ such that $\abs{\mathcal{M} \cap \omega_1} = 
	\aleph_0$. Then Chang's Conjecture holds. 
	\label{equivalent_formulation_claim}
\end{claim}

\begin{proof}
	This is a standard model-theoretic argument. Suppose that 
	$\mathcal{N} = \langle N, R_1, R_2, \dots \rangle$ is any 
	structure with $\omega_2 \subseteq N$. We may assume without 
	loss of generality that $\abs{\mathcal{N}} = \aleph_2$. 
	Let $\pi \colon N \rightarrow H\!\paren{\omega_2}$ be any 
	injection which is the identity on $\omega_2$. 
	Let $\mathcal{H} = \langle H\!\paren{\omega_2}, \tilde{N}, 
	\tilde{R}_1, \tilde{R}_2, \dots \rangle$, where 
	$\tilde{N}$ is a predicate representing membership in 
	$\pi\brac{N}$ and $\tilde{R}_i$ is a predicate representing 
	$R_i$ in the natural way. 
	By our assumption there is $\mathcal{M} \preceq 
	\mathcal{H}$ of cardinality $\aleph_1$ such that 
	$\abs{\mathcal{M} \cap \omega_1} = \aleph_0$. Pulling 
	back via $\pi$, we get the desired submodel of $\mathcal{N}$. 
\end{proof}

Square properties are a family of ``incompactness principles''
regarding sequences of clubs. 

\begin{defn}[\cite{combinatorial_principles_in_core_model}]
	Suppose that $\kappa$ is an infinite cardinal 
	and $\lambda$ is a nonzero (but potentially finite) cardinal. 
	A $\square_{\kappa, \lambda}$-sequence is a sequence 
	$\vec{\mathcal{C}} = \angles{\mathcal{C}_{\alpha} \colon 
\alpha < \kappa^+}$ such that: 
\begin{enumerate}
	\item For all $\alpha < \kappa^+$, $1 \leq \abs{\mathcal{C}_{\alpha}} \leq \lambda$. 
	\item For all $\alpha < \kappa^+$ and $C \in \mathcal{C}_{\alpha}$, 
		$C$ is a club subset of $\alpha$ and 
		$\otp{C} \leq \kappa$. 
	\item (Coherence) For all $\alpha < \kappa^+$, every $C \in \mathcal{C}_{\alpha}$
		\emph{threads} $\angles{\mathcal{C}_{\beta} \colon \beta < \alpha}$
		in the sense that $C \cap \beta \in \mathcal{C}_{\beta}$ for all 
		$\beta$ which are limit points of $C$. 
\end{enumerate}
We say that $\square_{\kappa, \lambda}$ holds if such a sequence exists. 
\end{defn}
In this section we will be concerned only with $\square_{\omega_1, 2}$. 
In order to obtain our result, we will need to make use of a large cardinal 
hypothesis: 

\begin{defn}
	A cardinal $\kappa$ is said to be \emph{$\omega_1$-Erd\H{o}s} if it is least such that for 
	any partition $f \colon \brac{\kappa}^{< \omega} \rightarrow 2$, there is 
	$H \in \brac{\kappa}^{\omega_1}$ which is homogeneous for $f$. 
\end{defn}

\begin{lem}[Silver]
	If $\kappa$ is $\omega_1$-Erd\H{o}s, then for any structure $\mathcal{M}$ with 
	$\kappa \subseteq \mathcal{M}$, there is a set of indiscernibles $I \in \brac{\kappa}^{\omega_1}$
	for $\mathcal{M}$. Morever, if $\mathcal{M}$ has underlying set $H\!\paren{\kappa}$ and 
	includes among its predicates some 
	$\vartriangleleft$ which is a well-ordering of its universe, 
	we may assume $I$ consists of inaccessible cardinals which are \emph{remarkable} in the 
	sense that for any $\gamma \in I$, $I\setminus \gamma$ is a set of 
	indiscernibles for $\langle\mathcal{M}, \paren{\delta}_{\delta < \gamma}\rangle$. 
\end{lem}
\begin{proof}
	See \cite{higher_infinite}, \cite{some_apps_of_core_model}. 
\end{proof}

\subsection{The Poset} \label{the_poset}
Our poset $\mathbb{P}$ is a ``Silverized'' version of the one appearing in 
\cite{indexed_squares} in the sense that we modify their poset 
to allow conditions with $\omega_1$-sized support. 
We define $\mathbb{P} = \mathbb{P}_\kappa$ as follows: set $p \in 
\mathbb{P}$ iff $p$ is a function so that

\begin{enumerate}[(1)]
	\item The domain of $p$ is a closed $\leq \omega_1$-sized 
		set of limit ordinals less than $\kappa$. 
	\item If $\cf{\alpha} = \omega$ and $\alpha \in \dom{p}$
		then $1 \leq \abs{p\paren{\alpha}} \leq 2$
		and each set in $p\!\paren{\alpha}$ is a club 
		subset of $\alpha$ with countable order type. 
	\item If $\cf{\alpha} = \omega_1$ and $\alpha \in 
		\dom{p}$ then $p\!\paren{\alpha} = 
		\left\{ C \right\}$ where $C$ is a club 
		subset of $\alpha$ with order type 
		$\omega_1$. 
	\item If $\cf{\alpha} \geq \omega_2$ then 
		$p\!\paren{\alpha} = \curly{C}$ where 
		$C$ is a closed bounded subset of 
		$\alpha$ with countable order type 
		such that $\max{C} =
		\sup{\paren{ \dom{p} \cap \alpha}}$. 
	\item If $\alpha \in \dom{p}$, $C \in p\!\paren{\alpha}$ 
		and $\beta \in \lim{(C)}$, then 
		$\beta \in \dom{p}$ and 
		$C \cap \beta \in p\!\paren{\beta}$. 
	\item The supremum of $\otp{C}$ taken over 
		all $C \in p\!\paren{\alpha}$, 
		$\cf{\alpha} \geq \omega_2$, is 
		strictly below $\omega_1$. 
\end{enumerate}

For two elements $p, q \in \mathbb{P}_\kappa$, we set 
$p \leq q$ iff: 
\begin{enumerate}
	\item $\dom{q} \subseteq \dom{p}$
	\item For all $\alpha \in \dom{q}$: 
		\begin{enumerate}[(a)]
			\item If $\cf{\alpha} \in \curly{\omega, \omega_1}$, 
				then $p\!\paren{\alpha} = q\!\paren{\alpha}$. 
			\item If $\cf{\alpha} \geq \omega_2$, 
				$p\!\paren{\alpha} = \curly{C}$ and 
				$q\!\paren{\alpha} = \curly{D}$, then 
				$C$ is an end-extension of $D$ in the 
				sense that $D = C \cap \paren{\max{(D)} + 1}$. 
		\end{enumerate}
\end{enumerate}

\begin{lem}
	Suppose that $\kappa$ is inaccessible. Then $\mathbb{P} = \mathbb{P}_\kappa$ is 
	$\kappa$-c.c. and countably closed, and collapses 
	$\kappa$ to $\aleph_2$ while adding a $\square_{\omega_1, 2}$-sequence. 
	\label{cc_and_closure_lem}
\end{lem}

\begin{proof}
	The proof is very similar to that of the corresponding result in 
	$\cite{indexed_squares}$. 
	The fact that $\mathbb{P}$ is $\kappa$-c.c. follows from a standard $\Delta$-system 
	argument. If we can show that $\mathbb{P}$ is countably closed, then the second 
	conclusion follows immediately. So suppose that $\angles{p_n \colon n < \omega}$
	is a decreasing sequence of conditions. 

	Let $X$ be the set of $\alpha \in \bigcup_{n < \omega}{\dom{p_n}}$ such that 
	the value of $p_n\!\paren{\alpha}$ does not eventually stabilize and let 
	\begin{align*}
		Y = \{\sup_{n < \omega}{ \max{p_n\!\paren{\alpha}}} \colon 
	\alpha \in X \}
	\end{align*}
	Observe that $Y \cap \paren{\bigcup_{n < \omega}{\dom{p_n}}} = \emptyset$, since 
	if $\alpha \in X$ the fact that $\max{(p_n\!\paren{\alpha})} \geq \sup{(\dom{p_n} \cap \alpha)}$
	for every $n$ gives 
	\begin{align*}	
	\sup_{n < \omega}{ \max{p_n\!\paren{\alpha}}} \notin 
	\bigcup_{n < \omega}{\dom{p_n}}
\end{align*}
	Let 
	\begin{align*}
		Z = \overline{ \paren{\bigcup_{n < \omega}{\dom{p_n}}}} \cup Y
	\end{align*}
	where the overline indicates closure in the ordinal topology. 
	We claim that $Z$ is closed. To show this it suffices to show that 
	any limit point of $Y$ lies in $\overline{\bigcup_{n < \omega}{\dom{p_n}}}$. 
	Moreover, this will itself follow from the assertion that any element of 
	$Y$ lies in $\overline{\bigcup_{n < \omega}{\dom{p_n}}}$. But this is immediate 
	by condition (4) in the definition of $\mathbb{P}$. 

	We will define a condition $p_{\omega}$	with domain 
	$Z$ which is a lower bound for $\langle p_n \colon n < \omega \rangle$. 
	First, if $\alpha \in \bigcup_{n < \omega}{\dom{p_n}} \setminus X$, 
	let $p_\omega\!\paren{\alpha}$ be the eventual value of the 
	sequence $\langle p_n\!\paren{\alpha} \colon n < \omega \rangle$. 
	If $\alpha \in X$, then set 
	\begin{align*}
		p_{\omega}\!\paren{\alpha} = \bigcup_{n < \omega}{p_n\!\paren{\alpha}} \cup 
		\{ \sup_{n < \omega}{ \max{p_n\!\paren{\alpha}}} \}
	\end{align*}
	Next, if $\alpha \in Y$ then $\alpha = \sup_{n < \omega}{\max{p_n\!\paren{\beta}}}$
	for a unique $\beta \in X$, and we set 
	\begin{align*}
		p_{\omega}\!\paren{\alpha} = \bigcup_{n < \omega}{p_n\!\paren{\beta}}
		\cup \{ \sup_{n < \omega}{\max{p_n\!\paren{\beta}}}\}
	\end{align*}
	for this $\beta$. Finally, suppose $\alpha \in 
	\overline{\paren{\bigcup_{n < \omega}{\dom{p_n}}}} \setminus \paren{\bigcup_{n < \omega}{\dom{p_n}}}$ 
	and $p_\omega(\alpha)$ is yet to be defined. Set
	\begin{align*}
		p_{\omega}(\alpha) = 
		\left\{ \max{\paren{ \dom{(p_n)} \cap \alpha}} \colon n < \omega \right\}
	\end{align*}
	Clearly this set is unbounded in $\alpha$. Moreover, this set has order-type $\omega$, and therefore 
	has no limit points below $\alpha$ (and is club in $\alpha$). Therefore we are in no danger of 
	violating coherence (condition (5) in the definition of $\mathbb{P}$) by defining 
	$p_\omega(\alpha)$ as such.

	We refer to the condition $p_\omega$ defined above as the \emph{canonical lower bound} of 
	$\angles{p_n \colon n < \omega}$. 
\end{proof}

We also define a \emph{threading poset} for a given $\square_{\omega_1, 2}$-sequence. 
Supposing that $\vec{\mathcal{C}} = \angles{C_\alpha \colon \alpha < \omega_2}$ is 
such a sequence, we let $\mathbb{T} = \mathbb{T}_{\vec{\mathcal{C}}}$ be the 
poset of closed bounded subsets $C$ of $\omega_2$ of countable order type  
such that $C$ threads $\angles{\mathcal{C}_\alpha \colon \alpha \leq \max{C}}$ 
in the sense that $C \cap \alpha \in \mathcal{C}_{\alpha}$ for all $\alpha$ which 
are limit points of $C$. 

If $C, D \in \mathbb{T}$, then we set $C \leq D$ if and only if $C$ is 
an end-extension of $D$. 

Finally, suppose that $\mu < \kappa$ are two inaccessible cardinals. 
If $G$ is the generic added by $\mathbb{P}_{\mu}$, then 
$\Q = \Q_{\mu, \kappa, G}$ is the poset in $V\!\brac{G}$ defined by 
setting $q \in \Q$ iff $q \in V$ and: 

\begin{enumerate}[(a)]
	\item $\dom{q}$ is a closed $\leq \omega_1$-sized set of limit ordinals 
		in the interval $\paren{\mu, \kappa}$. 
	\item If $\cf{\alpha} = \omega$ and $\alpha \in \dom{q}$, then 
		$1 \leq \abs{q\!\paren{\alpha}} \leq 2$ and 
		each element of $q\!\paren{\alpha}$ is a club 
		with countable order type. 
	\item If $\cf{\alpha} = \omega_1$ and $\alpha \in \dom{q}$ then 
		$q\!\paren{\alpha} = \curly{C}$ where $C$ is a club 
		subset of $\alpha$ with order type $\omega_1$. 
	\item If $\cf{\alpha} \geq \omega_2$, then $q\!\paren{\alpha} = \curly{C}$
		where $C$ is a closed bounded subset of $\alpha$ with 
		countable order type such that $\max{C} =
		\sup{(\dom{q} \cap \alpha)}$. 
	\item If $\alpha \in \dom{q}$, $C \in q\!\paren{\alpha}$, and 
		$\beta \in \lim{C}$, then:
		\begin{enumerate}[(A)]
	\item If $\beta > \mu$, then 
		$\beta \in \dom{q}$ and 
		$C \cap \beta \in q\!\paren{\beta}$. 
	\item If $\beta < \mu$, then $C \cap \beta \in \mathcal{C}_{\beta}$, 
		where $\angles{\mathcal{C}_\beta \colon \beta < \mu}$
		is $\bigcup{G}$. 
\end{enumerate}
	\item The supremum of $\otp{C}$ taken over all $C \in q\!\paren{\alpha}$, 
		$\cf{\alpha} \geq \omega_2$, is strictly below $\omega_1$. 
\end{enumerate}

For two elements $p, q \in \mathbb{Q}_{\mu, \kappa}$, we set 
$p \leq q$ iff: 
\begin{enumerate}[(1)]
	\item $\dom{q} \subseteq \dom{p}$
	\item For all $\alpha \in \dom{q}$: 
		\begin{enumerate}[(a)]
			\item If $\cf{\alpha} \in \curly{\omega, \omega_1}$, then 
				$p\!\paren{\alpha} = q\!\paren{\alpha}$. 
			\item If $\cf{\alpha} \geq \omega_2$, 
				$p\!\paren{\alpha} = \curly{C}$, $q\!\paren{\alpha} = \curly{D}$, 
				then $C$ is an end-extension of $D$. 
		\end{enumerate}
\end{enumerate}

\begin{claim}
	Suppose that $\mu$, $\kappa$ are inaccessible cardinals with 
	$\mu < \kappa$, and $\dot{G}$ is the canonical name for the 
	$\mathbb{P}_{\mu}$-generic. Then if we let $\dot{\mathbb{T}} = \check{\mathbb{T}}_{ \bigcup{\dot{G}}}$, 
	$\dot{\Q} = \check{\Q}_{\mu, \kappa, \dot{G}}$, there is an isomorphism between 
	a dense subset of $\mathbb{P}_{\kappa}$ and a dense subset of 
	$\mathbb{P}_{\mu} \ast \dot{\mathbb{T}} \ast \dot{\Q}$. In particular these two 
	forcings are equivalent, so informally we may view them as being equal. 
\end{claim}

\begin{proof}
	As in \cite{indexed_squares}. 
\end{proof}

\subsection{The Proof} \label{the_proof}

\begin{theorem}
	Suppose that $\kappa$ is an $\omega_1$-Erd\H{o}s cardinal. Let 
	$\mathbb{P} = \mathbb{P}_\kappa$. 
	Then for any $\mathbb{P}$-generic $G$, 
	$V\!\brac{G}$ satisfies Chang's Conjecture. 
	\label{consistency_theorem}
\end{theorem}

\begin{corollary}
	The existence of an $\omega_1$-Erd\H{o}s cardinal is 
	equiconsistent with ``Chang's Conjecture plus $\square_{\omega_1, 2}$.''
	\label{obvious_corollary}
\end{corollary}

\begin{proof}[Proof of Corollary \ref{obvious_corollary}]
	By Theorem \ref{consistency_theorem} 
	and Lemma \ref{cc_and_closure_lem} an $\omega_1$-Erd\H{o}s 
	cardinal suffices for the consistency of Chang's Conjecture 
	plus $\square_{\omega_1, 2}$. By \cite{some_apps_of_core_model}, 
	the consistency of Chang's Conjecture implies that of the 
	existence of of an $\omega_1$-Erd\H{o}s cardinal. 
\end{proof}

\begin{proof}[Proof of Theorem \ref{consistency_theorem}]
	Suppose that $G$ is a $\mathbb{P}$-generic over $V$. Then 
	$\omega_2^{V\brac{G}} = \kappa$ and 
	$\paren{H(\kappa)}^{V\brac{G}} = H(\kappa)\!\brac{G}$.
	Let $\mathcal{H} = \langle H(\kappa), \in, \dot{R} \rangle$, 
	which we view as a name for a structure $\mathcal{H}\!\brac{G}$
	with underlying set 
	$H(\kappa)\!\brac{G}$ and 
	predicate $R = \dot{R}^G \subseteq H(\kappa)\!\brac{G}$. 

	We seek a condition $p^* \in \mathbb{P}$ and a name $\dot{\mathcal{A}}$ for an 
	elementary substructure $\mathcal{A}$ of $\mathcal{H}\!\brac{G}
	$ such that 
	$p^*$ forces $| \dot{\mathcal{A}}| = \aleph_1$, 
	$|\mathcal{A} \cap \omega_1| = \aleph_0$. 
	With this in mind, let $I = \curly{\iota_\alpha \colon 
	\alpha < \omega_1}$ be a collection of remarkable 
	indiscernibles for $\mathcal{H}$.
	For each $\alpha < \omega_1$, let 
	$I_\alpha = \curly{\iota_\delta \colon \delta < \omega \alpha}$ be the 
	set of the first $\omega \alpha$ indiscernibles and let 
	$\gamma_\alpha = \iota_{\omega \alpha}$. Let 
	$\mathcal{M}_\alpha$ be the Skolem Hull of $I_\alpha$ in 
	$\mathcal{H}$. 
	
	We construct a sequence $\angles{p^*_\alpha \colon 1 \leq \alpha < \omega_1}$
	by induction on $\alpha$ so that: 
	\begin{enumerate}[(a)]
		\item If $1 \leq \alpha < \beta < \omega_1$ then $p_{\beta}^* \leq p_{\alpha}^*$. 
		\item $p_{\alpha}^*$ is a master condition for $\mathbb{P}$ over 
			$\mathcal{M}_\alpha$. 
		\item $p_{\alpha}^*$ is an element of $\mathbb{P}_{\gamma_\alpha}$. 
	\end{enumerate}

	We begin with the base case $\alpha = 1$.
	Consider the set $\mathbb{P} \cap \mathcal{M}_1 = \paren{\mathbb{P}
	_{\text{ON}}}^{\mathcal{M}_1}$, which is a proper class in $\mathcal{M}_1$. 
	Observe that since $\mathcal{M}_1$ is elementary in 
	$\mathcal{H}$, $\mathcal{M}_1$ satisfies ``$\mathbb{P}$ has the $<$-ON chain condition.''
	In other words, $\mathcal{M}_1$ believes that every antichain in 
	$\mathbb{P}$ is a set. For each antichain $A$ in $\mathcal{M}_1$, let 
	$A^{\downarrow} = \curly{p \in \mathbb{P} \colon \paren{\exists q \in A} p \leq q}$ 
	be the downwards closure of $A$. Let $\left\{ A_i \colon i < \omega \right\}$
	enumerate the collection of all maximal antichains which are elements of 
	$\mathcal{M}_1$. By induction we may construct a descending sequence 
	$\left\{ r_i \colon i < \omega \right\}$ of elements of 
	$\mathbb{P}$ such that $r_i \in A_i^{\downarrow} \cap \mathcal{M}_1$. 
	Let $p_1^* \in \mathbb{P}$ be the canonical lower bound for the 
	sequence $\left\{ r_i \colon i < \omega \right\}$. 
	Then $p_1^*$ is a master condition for $\mathbb{P}$
	over $\mathcal{M}_1$ and is an element of $\mathbb{P}_{\gamma_1}$, as desired. 

	Next suppose that $\alpha$ is limit. Choose a sequence $\angles{\alpha_n \colon n < \omega}$ 
	cofinal in $\alpha$, and let $p_{\alpha}^*$ be the canonical lower bound
	for $\angles{p^*_{\alpha_n} \colon n < \omega}$. It should be clear that 
	properties (a)-(c) are satisfied, since $\mathbb{P} \cap \mathcal{M}_\alpha = 
	\bigcup_{n < \omega}{\paren{\mathbb{P} \cap \mathcal{M}_{\alpha_n}}}$, and 
	$\mathbb{P} \cap \mathcal{M}_\alpha = \paren{\mathbb{P}_{\text{ON}}}^{\mathcal{M}_\alpha}$
	has the $<$-ON chain condition in $\mathcal{M}_\alpha$. 

	Finally we consider the case where $\alpha = \bar{\alpha} + 1$ is a successor 
	ordinal. We distinguish between the case where $\bar{\alpha}$ is a nonzero limit ordinal and 
	where $\bar{\alpha}$ is itself a successor ordinal, considering first the latter. 
	Since $p^*_{\bar{\alpha}}$ was chosen to be a master condition for $\mathbb{P}$ over 
	$\mathcal{M}_{\bar{\alpha}}$, we have
	
	\begin{align*}
		p_{\bar{\alpha}}^* \forces \mathcal{M}_{\bar{\alpha}}[\dot{G}] \preceq \mathcal{H}[\dot{G}]	
		\land \text{ON} \cap \mathcal{M}_{\bar{\alpha}}[\dot{G}] = \text{ON} \cap \mathcal{M}_{\bar{\alpha}}
	\end{align*}

	Consider $\mathcal{M}_\alpha$. By remarkability of the indiscernibles 
	which generate $\mathcal{M}_\alpha$, we have $H(\gamma_{\bar{\alpha}}) \cap \mathcal{M}_{\alpha} = 
	\mathcal{M}_{\bar{\alpha}}$ and $\mathbb{P}_{\gamma_{\bar{\alpha}}} \cap \mathcal{M}_{\alpha} = \mathbb{P} \cap \mathcal{M}_{
		\bar{\alpha}}$. 
		Moreover, $p^*_{\bar{\alpha}}$ is a master condition for the forcing 
		$\mathbb{P}_{\gamma_{\bar{\alpha}}}$ over the model 
		$\mathcal{M}_\alpha$, since $\mathbb{P}_{\gamma_{\bar{\alpha}}}$ has the $\gamma_{
			\bar{\alpha}}$-c.c. and
			therefore every antichain of $\mathbb{P}_{\gamma_{\bar{\alpha}}}$ in $\mathcal{M}_\alpha$
			is an element of $H(\gamma_{\bar{\alpha}}) \cap \mathcal{M}_{\alpha} = \mathcal{M}_{\bar{\alpha}}$. 
			So if we let $\dot{G}_{\gamma_{\bar{\alpha}}}$ be the canonical name for the 
			$\mathbb{P}_{\gamma_{\bar{\alpha}}}$-generic, then
	\begin{align*}
		p_{\bar{\alpha}}^* \forces \mathcal{M}_{\alpha}[\dot{G}_{\gamma_{\bar{\alpha}}}] \preceq 
		\mathcal{H}[\dot{G}_{\gamma_{\bar{\alpha}}}] \land \text{ON} \cap \mathcal{M}_\alpha[\dot{G}_{
		\gamma_{\bar{\alpha}}}] = 
		\text{ON} \cap \mathcal{M}_{\alpha}
	\end{align*}
	Working in $V$, let $\dot{\mathbb{T}} = \check{\mathbb{T}}_{\bigcup{\dot{G}_{\gamma_{\bar{\alpha}}}}}$ be the 
	canonical name for the threading 
	forcing associated to $G_{\gamma_{\bar{\alpha}}}$. Let 
	$\{ \dot{B}_i \colon i < \omega \}$ enumerate all names in $\mathcal{M}_{\alpha}$
	which are forced by $p_{\bar{\alpha}}^*$ to be maximal antichains of $\dot{\mathbb{T}}$. 
	By induction we may construct a 
	descending sequence $\left\{ \check{t}_i \colon i < \omega \right\}$ 
	of ``check-names'' (by which we mean canonical names for elements of $V$) 
	for elements of $\mathbb{T} = \dot{\mathbb{T}}^{G_{\gamma_{\bar{\alpha}}}}$ such that
	\begin{align*}	
		p_{\bar{\alpha}}^* \forces \check{t}_i \in \dot{B}_i^{\downarrow} \cap 
		\mathcal{M}_{\alpha}[ \dot{G}_{\gamma_{\bar{\alpha}}}]
	\end{align*}	
	where $\dot{B}_i^{\downarrow}$ is a name for the downwards closure of 
	$B_i = \dot{B}_{i}^{G_{\gamma_{\bar{\alpha}}}}$ in $\mathbb{T}$. Observe that we may take 
	canonical names for elements of $V$ $\check{t}_i$ rather than merely arbitrary names $\dot{t}_i$
	since $p_{\bar{\alpha}}^*$ is a master condition for $\mathbb{P}_{\gamma_{\bar{\alpha}}}$ over 
	$\mathcal{M}_{\alpha}$.

	Still working in $V$, we let 
	\begin{align*}
		t &= \bigcup_{i < \omega}{t_i} \\
		p_{\bar{\alpha}}^{**} &= p_{\bar{\alpha}}^* \cup \curly{ \paren{\sup{\paren{\mathcal{M}_{\bar{\alpha}}
	\cap \kappa}}, \curly{t}}}
	\end{align*}
	Then $p_{\bar{\alpha}}^{*} \ast \check{t}$ is a master condition for $\mathbb{P}_{\gamma_{\bar{\alpha}}} \ast \dot{\mathbb{T}}$ 
	over $\mathcal{M}_{\alpha}$. Now observe that 
	$\dot{\mathbb{Q}} = \check{\mathbb{Q}}_{\gamma_{\bar{\alpha}}, \text{ON}, \dot{G}_{\gamma_{\bar{\alpha}}}}$ 
	is definable over $\langle \mathcal{M}_{\alpha}[G_{\gamma_{\bar{\alpha}}}], \in, \mathcal{M}_\alpha \rangle$ 
	(i.e. the structure $\mathcal{M}_{\alpha}[G_{\gamma_{\bar{\alpha}}}]$ with signature expanded to include a 
	predicate for membership in $\mathcal{M}_{\alpha}$). So we may proceed as above to find 
	$\dot{q} \in \dot{\Q}$ such that 
	 $p_{\bar{\alpha}}^{*} \ast \check{t} \ast \dot{q}$ is a master condition for 
	 $\mathbb{P}_{\gamma_{\bar{\alpha}}} \ast \dot{\mathbb{T}} \ast \dot{\mathbb{Q}}$ over $\mathcal{M}_{\alpha}$. 
	 Thus if we set $p_{\alpha}^* = p_{\bar{\alpha}}^{**} \ast \check{t} \ast \dot{q}$, we may view 
	 $p_{\alpha}^*$ as a master condition for $\mathbb{P}$ over $\mathcal{M}_{\alpha}$ which 
	 extends $p_{\bar{\alpha}}^*$. We note that $p^*_{\alpha}\!\paren{\sup{(M_{\bar{\alpha}} \cap \kappa)}} = \curly{t}$. 

	 For nonzero limit $\bar{\alpha}$, the construction is exactly as above, 
	 except we modify $p^*_{\alpha}\!\paren{\sup{(M_{\bar{\alpha}} \cap \kappa)}}$ to 
	 be $\left\{ t, F \right\}$, where 
	 $t$ is a master condition for the threading poset associated to 
	 the generic for $\mathbb{P}_{\gamma_\alpha}$ (as above) and 
	 $F = \curly{\sup{ \paren{\mathcal{M}_{\delta} \cap \kappa}} \colon \delta < \bar{\alpha}}$, 
	 rather than merely taking $p^*_{\alpha}\!\paren{\sup{(M_{\bar{\alpha}} \cap \kappa)}}$ to 
	 be $\left\{ t \right\}$. 

	 Observe that this is the only place in the proof where we use the allowed ``two-ness'' of 
	 the square sequence. Moreover, in adding $F$ we preserve the coherence property since
	 its initial segments of limit length were put on the square sequence at earlier successor of limit stages. 
	 
	 Finally, at the end of the construction we set 
	 \begin{align*}
		 p^* = \bigcup_{\alpha < \omega_1}{p_\alpha^*} 
		 \cup \big\{ ( \sup{ (  \big(\bigcup_{\alpha < \omega_1}{
			 \mathcal{M}_\alpha} \big) \cap \kappa )}, F^* ) \big\}
	 \end{align*}
	 where $F^* = \curly{ \sup{ \paren{\mathcal{M}_\alpha \cap \kappa}} \colon 
 \alpha < \kappa}$. The construction ensures that this is a condition in $\mathbb{P} = \mathbb{P}_\kappa$. In particular,
 successor of limit stages ensure that the initial segments of limit length of $F^*$ appear on the square sequence, and so when 
 adding $F^*$ there is no danger of violating coherence. 
 Moreover, $p^*$
 is a master condition for $\mathbb{P}$ over $\mathcal{M} = \bigcup_{\alpha < \omega_1}{
	 \mathcal{M}_\alpha}$. Thus $p^*$ forces that $\mathcal{M}\!\brac{G}$ is the desired 
	 elementary submodel of $\mathcal{H}\!\brac{G}$. 
\end{proof}


\section{Higher Chang's Conjectures vs. Weak Squares}
\label{second_section}
In this section we concern ourselves with generalizations of Chang's Conjecture to 
higher cardinals. 

\begin{defn}
	Suppose that $\tau \leq \kappa < \lambda$ are cardinals. We write 
	$(\lambda^+, \lambda) \twoheadrightarrow (\kappa^+, \kappa)$ if 
	for every structure $\mathcal{N}$ with $\lambda^+ \subseteq \mathcal{N}$, 
	there exists $\mathcal{M} \preceq \mathcal{N}$ such that 
	$| \mathcal{M} | = \kappa^+$ and $|\mathcal{M} \cap \lambda| = \kappa$. 

	Similarly, we write $(\lambda^+, \lambda) \twoheadrightarrow_{\tau} \paren{\kappa^+, \kappa}$ if
	for every structure $\mathcal{N}$ with $\lambda^+ \subseteq \mathcal{N}$, 
	there exists $\mathcal{M} \preceq \mathcal{N}$ such that 
	$| \mathcal{M} | = \kappa^+$, $| \mathcal{M} \cap \lambda | = \kappa$, and 
	$\tau \subseteq \mathcal{M}$. 
\end{defn}

Observe that Chang's Conjecture is equivalent to $(\aleph_2, \aleph_1) \twoheadrightarrow 
(\aleph_1, \aleph_0)$ and that $\paren{\lambda^+, \lambda} \twoheadrightarrow \paren{\kappa^+, \kappa}$ 
is equivalent to $\paren{\lambda^+, \lambda} \twoheadrightarrow_{\omega} \paren{\kappa^+, \kappa}$
for any infinite cardinals $\kappa < \lambda$. Moreover, we also have: 

\begin{lem}
	Suppose that $\tau \leq \kappa < \lambda$ are infinite cardinals and there are at most 
	$\tau$ many cardinals between $\kappa$ and $\lambda$. Then $(\lambda^+, \lambda) 
	\twoheadrightarrow_{\tau} (\kappa^+, \kappa)$ implies $(\lambda^+, \lambda) \twoheadrightarrow_{\kappa}
	(\kappa^+, \kappa)$. 
	\label{equivalence_of_cc_varieties1}
\end{lem}	

\begin{proof}
	The lemma is implicit in \cite{cc_weak_square}. Specifically, the conclusion of 
	the lemma holds by following the argument of Case (2) of Lemma 4.15 in 
	\cite{cc_weak_square}. 
\end{proof}

\begin{lem}
	Suppose that $\tau \leq \kappa < \lambda$ are infinite cardinals such that 
	$\lambda^\tau = \lambda$. Then $(\lambda^+, \lambda) \twoheadrightarrow 
	(\kappa^+, \kappa)$ implies $(\lambda^+, \lambda) \twoheadrightarrow_{\tau} 
	(\kappa^+, \kappa)$. 
	\label{equivalence_of_cc_varieties2}
\end{lem}

\begin{proof}
	Take $B = \tau$ in Case (1) of Lemma 4.15 in \cite{cc_weak_square}. 
\end{proof}

In the argument below we make use of the following claim 
without comment: 

\begin{claim}
	Suppose that for all sufficiently large $\theta$ and all 
	structures $\mathcal{H} = \langle H(\theta), \in, \dots \rangle$
	there exists $\mathcal{M} \preceq \mathcal{H}$ such that 
	$|\mathcal{M} \cap \lambda^+ | = \kappa^+$, $| \mathcal{M} \cap 
	\lambda | = \kappa$, and $\tau \subseteq \mathcal{M}$. 
	Then $(\lambda^+, \lambda) \twoheadrightarrow_{\tau} (\kappa^+, \kappa)$. 
\end{claim}

The proof is entirely analogous to that of Claim \ref{equivalent_formulation_claim}. 

\begin{lem}[Folklore]
	Suppose that $\kappa < \lambda$ are infinite cardinals and $\theta$ is a sufficiently 
	large regular cardinal. Let $M$ be an elementary substructure of $\angles{H\!\paren{\theta}, \in}$
	such that $\abs{M \cap \lambda^+} = \kappa^+$ and $\abs{M \cap \lambda} = \kappa$. Then 
	the order type of $M \cap \lambda^+$ is $\kappa^+$. 
	\label{otp_chang_substr_lem}
\end{lem}

\begin{proof}
	Suppose otherwise for a contradiction. Since $\abs{M \cap \lambda^+} = \kappa^+$, 
	the order type of $M \cap \lambda^+$ must be strictly greater than 
	$\kappa^+$. Let $\alpha$ be the $\kappa^{+}$ element of $M \cap \lambda^+$. Observe that 
	$\alpha \geq \lambda$ (since there are only $\kappa$ many elements of $M$ below $\lambda$)
	and hence by elementarity $\lambda = \abs{\alpha}$ is an element of $M$. Applying elementarity 
	again, there is $f \in M$ which is a bijection from $\alpha$ to $\lambda$. In particular, 
	\begin{align*}
		f``\paren{M \cap \alpha} \subseteq M \cap \lambda
	\end{align*}
	which is a contradiction since the left hand side has cardinality $\kappa^+$ (since $f$ is a bijection)
	whereas the right hand side has cardinality $\kappa$. 
\end{proof}

\begin{lem}[Folklore]
	Suppose that $M$ is an elementary substructure of $\angles{H\!\paren{\theta}, \in}$ for some 
	sufficiently large $\theta$ and $\alpha \in M$. Letting $\mu = \cf{\alpha}$, if 
	$f \in M$ is an increasing function from $\mu$ into $\alpha$ whose range is 
	cofinal in $\alpha$, then 
	\begin{align*}
		\sup{(f``\paren{M \cap \mu})} = \sup{(M \cap \alpha)}
	\end{align*} 
	\label{image_cofinal_in_sup_lemma}
\end{lem}

\begin{proof}
	Clearly $\sup{(f``(M \cap \mu))} \leq \sup{(M \cap \alpha)}$, since $f \in M$ and 
	$f \colon \mu \rightarrow \alpha$. For equality, suppose for a 
	contradiction that 
\begin{align*}
	\sup{(f``(M \cap \mu))} < \sup{(M \cap \alpha)}
\end{align*}
and choose $\beta \in M \cap \alpha$ such that $\beta > \sup{\paren{f``\paren{M \cap \alpha}}}$. 
By elementarity 
\begin{align*}
	M \models \paren{\exists \,\xi \in \mu}\paren{f \!\paren{\xi} > \beta}
\end{align*}
and so choosing $\xi_0 \in M \cap \mu$ to witness the existential statement above we have: 
\begin{align*}
	\sup{(f``(M \cap \mu))}	< \beta < f\!\paren{\xi_0}
\end{align*}
an obvious contradiction. 
\end{proof}

\begin{theorem}
	Suppose that $\kappa < \lambda$ are uncountable cardinals and $\tau \leq \kappa$ is infinite. 
	Suppose moreover that $\square_{\lambda, \tau}$ holds. Then 
	$\paren{\lambda^+, \lambda} \twoheadrightarrow_{\tau}
	\paren{\kappa^+, \kappa}$ fails.
	\label{incompatibility_theorem}
\end{theorem}

\begin{corollary}
	Suppose that $\kappa < \lambda$ are uncountable cardinals and $(\lambda^+, \lambda) 
	\twoheadrightarrow (\kappa^+, \kappa)$
	holds. Then $\square_{\lambda, \omega}$ fails. 
	\label{zeroeth_big_corollary}
\end{corollary}

\begin{proof}
	Immediate from the theorem and the fact that $\paren{\lambda^+, \lambda} \twoheadrightarrow 
	\paren{\kappa^+, \kappa}$ is equivalent to $\paren{\lambda^+, \lambda} \twoheadrightarrow_{\omega}
	\paren{\kappa^+, \kappa}$. 
\end{proof}

\begin{corollary}
	Suppose that $\kappa < \lambda$ are uncountable cardinals and there are at most 
	countably many cardinals between $\kappa$ and $\lambda$. 
	Then $\paren{\lambda^+, \lambda} \twoheadrightarrow \paren{\kappa^+, \kappa}$ implies the 
	failure of $\square_{\lambda, \kappa}$. 
	\label{first_big_corollary}
\end{corollary}

\begin{proof}
	This follows immediately from Theorem \ref{incompatibility_theorem} and Lemma \ref{equivalence_of_cc_varieties1} 
	by taking $\tau = \omega$. 
\end{proof}

Observe that the same argument shows that if there are at most $\tau$ many cardinals between 
$\kappa$ and $\lambda$ then $(\lambda^+, \lambda) \twoheadrightarrow_{\tau} \paren{\kappa^+, \kappa}$ 
implies the failure of $\square_{\lambda, \kappa}$. 

\begin{corollary}
	Suppose that $\kappa < \lambda$ are uncountable cardinals and $\tau \leq \kappa$ is some 
	infinite cardinal with $\lambda^\tau = \lambda$. Then 
	$(\lambda^+, \lambda) \twoheadrightarrow (\kappa^+, \kappa)$ implies the failure of 
	$\square_{\lambda, \tau}$. 
	\label{second_big_corollary}
\end{corollary}

\begin{proof}
	Immediate from Theorem \ref{incompatibility_theorem} and Lemma \ref{equivalence_of_cc_varieties2}. 
\end{proof}

\begin{proof}[Proof of Theorem \ref{incompatibility_theorem}]
	Suppose for a contradiction that $\square_{\lambda, \tau}$ held in conjunction with 
	$\paren{\lambda^+, \lambda} \twoheadrightarrow_{\tau} \paren{\kappa^+, \kappa}$, and let 
	$\vec{ \mathcal{C}} = \angles{ \mathcal{C}_{\xi} \colon \xi < \lambda^+}$ be a 
	$\square_{\lambda, \tau}$ sequence. Choose $M$ an elementary substructure of 
	$\langle H\!\paren{\theta}, \in, \vec{ \mathcal{C}} \rangle$ (for sufficiently large 
	$\theta$) such that $|M \cap \lambda^+ |
	= \kappa^+$, $| M \cap \lambda| = \kappa$, and $\tau \subseteq M$. 

	Fix a club $C^* \in \mathcal{C}_{\sup{(M \cap \lambda^+)}}$.
	By Lemma \ref{otp_chang_substr_lem}, we may choose a club $D$ in 
	$\sup{(M \cap \lambda^+)}$ of ordertype $\kappa^+$. We assume 
	moreover that $D$ consists only of limits of ordinals in $M$.

\begin{claim}
	For all sufficiently large $\alpha \in C^*$, the ordertype of 
$C^* \cap \alpha$ is not an element of $M$. 
	\label{ordertypes_not_in_m}
\end{claim}

\begin{proof}
	These ordertypes are distinct elements of $\lambda$, and since $| M \cap \lambda | 
	= \kappa$, 
	at most $\kappa$ of them can belong to $M$. Since the cofinality of 
	$\sup{C^*} = \sup{(M \cap \lambda^+)}$ is $\kappa^+$, the result follows immediately. 
\end{proof}

\begin{claim}
	For all sufficiently large $\alpha \in \Lim{C^*}$, $\alpha \notin M$. 
	\label{ordinals_not_in_m}
\end{claim}

\begin{proof}
	Choose $\alpha \in \Lim{C^*}$ and note that $C^* \cap \alpha \in \mathcal{C}_{\alpha}$. 
	If $\alpha \in M$, then $\mathcal{C}_{\alpha} \subseteq M$ (since $| \mathcal{C}_{\alpha} | \leq \tau$ and 
	$\tau \subseteq M$) and so in particular  $C^* \cap \alpha \in M$, giving $\otp{(C^* \cap \alpha)} \in M$. 
	By Claim \ref{ordertypes_not_in_m}, this may happen for only boundedly many 
	$\alpha \in C^*$. 
\end{proof}

For each $\alpha$ below $\sup{(M \cap \lambda^+)}$, let $\alpha^{\uparrow}$ denote the least element of $M$ which is 
$\geq \alpha$. 

\begin{claim}
	For all sufficiently large $\alpha \in \Lim{(C^* \cap D)}$, 
	$\alpha^{\uparrow}$ is strictly greater than $\alpha$. 
	\label{alpha_prime_bigger}
\end{claim}

\begin{proof}
	Immediate from Claim \ref{ordinals_not_in_m}. 
\end{proof}

Now define: 

\begin{align*}
	Z = \curly{\mu \leq \lambda \colon \mu = \cf{(\alpha^{\uparrow})} \text{ for unboundedly many $\alpha$ in 
		$\Lim{(C^* \cap D)}$}}
\end{align*}

\begin{claim}
	$|Z| \leq \kappa$.
	\label{cardinality_of_z}
\end{claim}

\begin{proof}
	For each $\alpha \in \Lim{(C^* \cap D)}$, $\alpha^{\uparrow}$ is an 
	element of $M$ below $\lambda^+$, and therefore its 
	cofinality is an element of $M \cap \paren{\lambda + 1}$, which 
	has cardinality $\kappa$. 
\end{proof}

	\begin{claim}
		There is $\mu \in Z$ with $\mu \geq \kappa^+$. 
	\end{claim}

	\begin{proof}
		By Claim \ref{cardinality_of_z}, it is enough to find unboundedly 
		many $\alpha \in \Lim{(C^* \cap D)}$ such that 
		$\cf{(\alpha^{\uparrow})} \geq \kappa^+$. 

		Fix any $\alpha \in \Lim{(C^* \cap D)}$ large enough for 
		Claim \ref{alpha_prime_bigger}, with 
		$\cf{(\alpha)} = \kappa$. Observe that there are unboundedly many 
		such $\alpha$ since the ordertype of $C^* \cap D$ is $\kappa^+$. 
		By choice of $\alpha$, $\sup(M \cap \alpha^{\uparrow}) = \alpha < \alpha^{\uparrow}$. 
		Then: 
		\begin{align*}
			\kappa = \cf{\alpha} < 
			 \cf{\alpha^{\uparrow}}
		\end{align*}
		by Lemma \ref{image_cofinal_in_sup_lemma}. 
	\end{proof}

	\begin{claim}
		$| Z | \geq 2$. 
	\end{claim}

\begin{proof}
	By Claim \ref{cardinality_of_z}, it suffices to find disjoint $A_1, A_2 \subseteq \Lim{(C^* \cap D)}$
	such that $A_1$, $A_2$ are unbounded and for any $\alpha_1 \in A_1$, $\alpha_2 \in A_2$, we have 
	$\cf{(\alpha_1^{\uparrow})} \neq \cf{(\alpha_2^{\uparrow})}$. 

	To do so, choose distinct regular $\eta_1, \eta_2 \leq \kappa$. Observe that this 
	is possible since $\kappa$ is uncountable. Now let 
	\begin{align*}
		A_1 &= \curly{\alpha \in \Lim{(C^* \cap D)} \colon \cf{\alpha} = \eta_1} \\
		A_2 &= \curly{\alpha \in \Lim{(C^* \cap D)} \colon \cf{\alpha} = \eta_2}
	\end{align*}
	Clearly $A_1$, $A_2$ are disjoint and unbounded. Moreover, for any 
	$\alpha_1 \in A_1$, $\alpha_2 \in A_2$, we have $\cf{(\alpha_1^{\uparrow})} \neq 
	\cf{(\alpha_2^{\uparrow})}$ by Lemma \ref{image_cofinal_in_sup_lemma}. 
\end{proof}

Now to prove the theorem: 

Fix distinct $\mu_1, \mu_2 \in Z$ with $\mu_1 > \mu_2$ and $\mu_1 \geq \kappa^+$. 
Fix $\alpha_1, \alpha_2 \in \Lim{(C^* \cap D)}$, large enough for Claims 
\ref{ordertypes_not_in_m} and \ref{ordinals_not_in_m}
and with $\alpha_1 < \alpha_2$, so that 
$\cf{(\alpha_1^{\uparrow})} = \mu_1$ and $\cf{(\alpha_2^{\uparrow})} = \mu_2$. 
Fix $E \in M$ cofinal in $\alpha_2^{\uparrow}$ of ordertype $\mu_2$. 

Let 
\begin{align*}
	U = \big\{ \sup{(C \cap \alpha_1^{\uparrow}) + 1 \colon C \in \bigcup_{\xi \in E}{\mathcal{C}_{\xi}} \text{ with }
	\sup{(C \cap \alpha_{1}^{\uparrow})} < \alpha_1^{\uparrow} } \big\}
\end{align*}

\begin{claim}
	$\sup{(C^* \cap \alpha^{\uparrow}_{1}) + 1} \in U$. 
	\label{final_claim}
\end{claim}

\begin{proof}
	Note first that $C^* \cap \alpha^{\uparrow}_{1}$ is bounded in $\alpha^{\uparrow}_{1}$, since 
	otherwise we would have $\alpha^{\uparrow}_1 \in \Lim{C^*}$, and as 
	$\alpha_1^{\uparrow} > \alpha_1$ is an element of $M$ this would contradict choice 
	of $\alpha_1$. Now since $E$ is club in $\alpha^{\uparrow}_{2}$ and belongs to 
	$M$, we have $\alpha_2 = \sup{(M \cap \alpha_2^{\uparrow})} \in E$, where the equality 
	follows from Lemma \ref{image_cofinal_in_sup_lemma}. Since 
	$C^*\cap \alpha_2 \in \mathcal{C}_{\alpha_2}$, and since 
	$C^* \cap \alpha_2 \cap \alpha^{\uparrow}_1 = C^* \cap \alpha^{\uparrow}_1$ is bounded in 
	$\alpha^{\uparrow}_1$, it follows by definition that $\sup{(C^* \cap \alpha^{\uparrow}_1)} + 1$ is 
	in $U$. 
\end{proof}

We have $U \in M$ by elementarity and since the parameters used are in $M$. 
$U$ has cardinality $< \mu_1$ by definition and since $\mu_1 > \max{(\mu_2, \kappa)}$. 
Since $\cf{(\alpha^{\uparrow}_1)} = \mu_1$, it follows that $U$ is bounded in 
$\alpha^{\uparrow}_{1}$. 

Moreover, since $U \in M$ we have $\sup{U} \in M$, and since 
$\sup{U} < \alpha^{\uparrow}_{1}$ it follows that 
$\sup{U} \leq \alpha_1$. But this contradicts Claim \ref{final_claim}, 
since $\alpha_1 \in C^*$ and therefore
\begin{align*}
	\sup{(C^* \cap \alpha^{\uparrow}_{1}) + 1 \geq \alpha_{1} + 1 > \alpha_1}
\end{align*}

\end{proof}

\medskip

\end{document}